\documentclass[12pt]{amsart}
\usepackage[active]{srcltx}
\usepackage{calc,amssymb,amsthm,amsmath,amscd, eucal,ulem}
\usepackage{alltt}
\usepackage[left=1.35in,top=1.25in,right=1.35in,bottom=1.25in]{geometry}
\RequirePackage[dvipsnames,usenames]{color}

\input{kmacros3.sty}
\input{xy}
\xyoption{all}
\usepackage{tikz}

\usepackage{upgreek}

\numberwithin{equation}{theorem}

\renewcommand{\m}{\mathfrak{m}}

\DeclareMathOperator{\perfd}{perfd}

\usepackage{fullpage}

\usepackage{setspace}
\usepackage{hyperref}
\usepackage{enumerate}
\usepackage{graphicx}

\usepackage[all,cmtip]{xy}
\usepackage{verbatim}

\theoremstyle{theorem}

\newtheorem*{theoremA*}{Theorem A}
\newtheorem*{theoremB*}{Theorem B}
\newtheorem*{theoremC*}{Theorem C}
\newtheorem*{theoremD*}{Theorem D}
\newtheorem*{theoremE*}{Theorem E}
\newtheorem*{theoremF*}{Theorem F}
\newtheorem*{theoremG*}{Theorem G}
\newtheorem*{theoremH*}{Theorem H}
\newtheorem*{Conjecture*}{Conjecture}

\begin{document}
\title{On complete integral closedness of the $p$-adic completion of absolute integral closure}
\author{Raymond Heitmann and Linquan Ma}
\address{Department of Mathematics\\University of Texas at Austin\\ Austin \\ TX 78712}
\address{Department of Mathematics\\ Purdue University\\ West Lafayette\\ IN 47907}
\email{heitmann@math.utexas.edu}
\email{ma326@purdue.edu}
\maketitle

\begin{abstract}
Fix a prime $p$ and let $(R,\m)$ be a Noetherian complete local domain of mixed characteristic $(0,p)$ with fraction field $K$. Let $R^+$ denote the absolute integral closure of $R$, which is the integral closure of $R$ in an algebraic closure $\overline{K}$ of $K$. The first author has shown that $\widehat{R^+}$, the $p$-adic completion of $R^+$, is an integral domain. In this paper, we prove that $\widehat{R^+}$ is completely integrally closed in $\widehat{R^+}\otimes_{R^+}\overline{K}$, but $\widehat{R^+}$ is not completely integrally closed in its own fraction field when $\dim(R)\geq 2$.
\end{abstract}

\section{Introduction}
Fix a prime number $p$ and let $(R,\m)$ denote a Noetherian complete local domain of mixed characteristic $(0,p)$. Recall that the absolute integral closure of $R$ is the integral closure of $R$ in an algebraic closure of its fraction field; such a closure is unique up to (non-unique) isomorphism and is typically denoted by $R^+$.  

In the last decade, there has been significant progress in the construction of big Cohen-Macaulay algebras in the mixed characteristic setting \cite{AndreDirectSummandConjecture,HeitmannMaBCMalgebrasVCT,GabberMSRINotes,AndreWeakFuncorialBCM,BhattCMnessAbsotluteIntegralClosure}. Most notably, the breakthrough result of Bhatt shows that $\widehat{R^+}$, the $p$-adic completion of $R^+$,  is big Cohen-Macaulay \cite{BhattCMnessAbsotluteIntegralClosure}. This result has found many applications in the study of birational geometry in mixed characteristic \cite{BMPSTWWMMP,BMPSTWWTestIdeal}. Due to its fundamental importance, it is natural to ask what ring-theoretic properties $\widehat{R^+}$ enjoys. 

It is not hard to show that $\widehat{R^+}$ is reduced: in fact, $\widehat{R^+}$ is a perfectoid ring in the sense of \cite{BhattMorrowScholzeIntegralPadicHodge}, and it is well-known that perfectoid rings are always reduced. However, in general, understanding the prime spectrum of a perfectoid ring can be difficult. Nonetheless, in this direction, a surprising result of the first author shows that $\widehat{R^+}$ is indeed an integral domain \cite{HeitmannAICSurprisinglyDomain}. A natural follow-up question is whether $\widehat{R^+}$ is (completely) integrally closed. Our main contribution towards this question is the following:

\begin{theoremA*}
Let $(R,\m)$ be a Noetherian complete local domain of mixed characteristic $(0,p)$ with fraction field $K$ and its algebraic closure $\overline{K}$. Then $\widehat{R^+}$ is completely integrally closed in $\widehat{R^+}\otimes_{R^+}\overline{K}$, while $\widehat{R^+}$ is completely integrally closed in its own fraction field if and only if $\dim(R)=1$. 
\end{theoremA*}

To study the (complete) integral closedness of an integral domain $S$ inside its fraction field, one can reduce the question to studying the (complete) integral closedness of $S$ inside its principal localization $S[1/g]$ for each nonzero element $g\in S$. In our case, we can characterize exactly when $\widehat{R^+}$ is completely integrally closed in $\widehat{R^+}[1/g]$ in terms of the $p$-adic topology of the quotient ring $\widehat{R^+}/g$.

\begin{theoremB*}
Let $(R,\m)$ denote a Noetherian complete local domain of mixed characteristic $(0,p)$. For a nonzero element $g\in \widehat{R^+}$, the following conditions are equivalent:
\begin{itemize}
  \item[(a)] $\widehat{R^+}/g$ is $p$-adically separated.
 % \item $\widehat{R^+}/g$ is $p$-adically complete;
 % \item[(2)] $v_Q(g)<\infty$ for all height one primes $Q\subseteq R^+$ that contain $p$;
 % \item[(2')] $\exists N>0$ such that $v_Q(g)<N$ for all height one primes $Q\subseteq R^+$ that contain $p$;
 % \item $\widehat{R^+}/g$ has bounded $p$-power torsion;
  \item[(b)] $\widehat{R^+}$ is completely integrally closed in $\widehat{R^+}[1/g]$.
\end{itemize}
\end{theoremB*}

In fact, in our Theorem~\ref{theorem: completeIntegrallyClosed}, we provide several other conditions, all equivalent to (a) and (b) above.  It is not hard to show that if $g\in R^+$, then $\widehat{R^+}/g$ is $p$-adically separated. However, when $\dim(R)\geq 2$, we construct explicit elements $g\in\widehat{R^+}$ such that $\widehat{R^+}/g$ is not $p$-adically separated (see Proposition~\ref{proposition: construction of bad g}). Thus Theorem A follows directly from Theorem B together with such construction.

This article is organized as follows: In Section 2 we recall some basic concepts and results that will be relevant. In Section 3 we collect some results on perfectoid rings, big Cohen-Macaulay algebras, and mixed characteristic closure operations. We prove Lemma~\ref{lemma: epf is p almost in ideal} on extended plus closure that will be needed in the proof of (a)$\Rightarrow$(b) in Theorem B. Finally in Section 4, we prove Theorem A and Theorem B.

\subsection*{Acknowledgement} The second author was supported by NSF Grant DMS \#2302430 and \#2424441. This material is based upon work supported by the NSF under Grant No. DMS-1928930 and by the Alfred P. Sloan Foundation under grant G-2021-16778, while the second author was in residence at the Simons Laufer Mathematical Sciences Institute (formerly MSRI) in Berkeley, California, during the Spring 2024 semester. The authors would like to thank Dimitri Dine and the referee for their comments on this article.

\section{Preliminaries}
Throughout this paper, all rings are commutative with multiplicative identity $1$. In this section, we collect some basic notations and results that will be used.

\subsection{Completion} Recall that if $J$ is an ideal in a ring $S$ and $M$ is an $S$-module, then the {\it $J$-adic completion} of $M$ is defined as $\varprojlim_n M/J^nM$. We are primarily focused in the case where $J=(p)$, the principal ideal generated by $p$, and $M=S$. We will use $\widehat{S}$ to denote the $p$-adic completion of $S$ throughout this article. Note that $\widehat{S}$ is always $p$-adically complete \cite[\href{https://stacks.math.columbia.edu/tag/05GG}{Tag 05GG}]{stacks-project}; but beyond the Noetherian setting, the quotient of a $p$-adically complete ring may not be $p$-adically complete. In fact, if $I$ is an ideal of a $p$-adically complete ring $S$, then $S/I$ need not be $p$-adically separated and the closure of $I$ in the $p$-adic topology of $S$ is $\cap_n(I, p^n)$. 

There is a closely related notion called {\it derived completion}. We briefly recall the definition of derived $p$-completion. Given any Abelian group $M$, the derived $p$-completion of $M$ can be defined as $R\varprojlim_n M\otimes^L_{\mathbb{Z}}\mathbb{Z}/p^n$. We refer the reader to \cite[\href{https://stacks.math.columbia.edu/tag/091N}{Tag 091N}]{stacks-project} for more general definitions and properties of derived completion. Here we simply point out that when $M$ has bounded $p$-power torsion, i.e., $\Gamma_{(p)}M= 0:_M p^N$ for some $N\in\mathbb{N}$, then derived $p$-completion of $M$ agrees with the usual $p$-adic completion of $M$, see \cite[\href{https://stacks.math.columbia.edu/tag/0BKF}{Tag 0BKF}]{stacks-project}.

\subsection{Complete integral closure}
%Let $S\to T$ be a ring extension. An element $t\in T$ is {\it integral} over $S$ if $t$ satisfies a monic equation with coefficients in $S$. 

Let $S\subseteq T$ be an extension of integral domains with the same fraction field. An element $t\in T$ is {\it almost integral} over $S$ if there exists a nonzero $s\in S$ so that $st^n\in S$ for all $n\in \mathbb{N}$. The set of all elements in $T$ almost integral over $S$ forms a ring and is called the {\it complete integral closure} of $S$ in $T$. We say $S$ is completely integrally closed in $T$ if its complete integral closure in $T$ agrees with $S$. 
If $S$ is Noetherian, then every almost integral element is integral over $S$ and thus the complete integral closure of $S$ in $T$ agrees with the usual integral closure. We refer the readers to \cite[Exercise 2.26]{SwansonHunekeIntegralClosure} for other basic properties of complete integral closure.

\begin{remark}
For an arbitrary ring extension $S\subseteq T$, one can consider the set of all elements $t\in T$ such that $\{t^n\}_{n\in\mathbb{N}}$ is contained in a finitely generated $S$-submodule of $T$. It is straightforward to check that this set forms a ring $S'$ and we have $S\subseteq S'\subseteq T$.\footnote{In \cite{GilmerHeinzerCompleteIntegralClosure}, $S'$ is called the complete integral closure of $S$ in $T$, and in \cite[Chapter 5]{BhattPerfectoidNotes}, $S'$ is called the total integral closure of $S$ in $T$.}
When $S$ is an integral domain and $T$ is the fraction field of $S$, it is easy to see that $S'$ agrees with the complete integral closure we defined above. However, in general they
%SSP
%
 are different even when $S$ is an integral domain and $T$ is a subring of the fraction field of $S$. For example consider $S=\mathbb{Z}_p+x\mathbb{Q}_p[x]\subseteq T:=S[1/p]$; then $1/p\in T$ is almost integral over $S$ (as $x/p^n\in S$ for all $n\in\mathbb{N}$) and hence $1/p$ is in the complete integral closure of $S$ in $T$ as we defined above, but $1/p\notin S'$ (as there is no fixed power of $p$ that multiplies $1/p^n$ into $S$ for all $n\in\mathbb{N}$).  
\end{remark}

\subsection{Absolute integral closure} Let $S$ be an integral domain with fraction field $L$. The absolute integral closure of $S$, denoted by $S^+$, is the integral closure of $S$ in an algebraic closure $\overline{L}$ of $L$ (or equivalently, $S^+$ is the direct limit of all module-finite domain extensions of $S$ in $\overline{L}$). %Since $S^+$ is integrally closed, it follows from \cite[Proposition 1.5.2]{SwansonHunekeIntegralClosure} that principal ideals are integrally closed.
We are primarily interested in the case that $S=R$ is a Noetherian complete local domain of mixed characteristic $(0,p)$. In this case, it is well-known (which essentially follows from Hensel's lemma) that $R^+$ is local. Thus its $p$-adic completion $\widehat{R^+}$ is also local. We will use $\m_{R^+}$ and $\m_{\widehat{R^+}}$ to denote the unique maximal ideals of $R^+$ and $\widehat{R^+}$ respectively. In \cite{HeitmannAICSurprisinglyDomain}, the first author proves that $\widehat{R^+}$ is an integral domain. The next lemma follows from \cite[Lemma 1.4]{HeitmannAICSurprisinglyDomain}, but we give a short and self-contained proof.

\begin{lemma}
\label{lemma: v(g) is not infinity}
Let $(R,\m)$ be a Noetherian complete local domain of mixed characteristic $(0,p)$. Then for any nonzero element $g\in\widehat{R^+}$, there exists an $\mathbb{R}$-valuation $v$ on $\widehat{R^+}$ centered at $\m_{\widehat{R^+}}$ such that $v(g)\neq \infty$ (i.e., $g$ is not in the support of $v$). 
\end{lemma}
\begin{proof}
Let $A:=C_k[[x_2,\dots,x_d]]\hookrightarrow R$ be a Noether-Cohen normalization of $R$ (where $C_k$ is a coefficient ring of $R$ and $k$ denotes the residue field of $R$). Without loss of generality, we may assume that $p \nmid g$ and we write $g=g_0+ph$ where $g_0\in R^+$ and $h\in\widehat{R^+}$. Let 
$$T^n+ a_1T^{n-1}+ \cdots + a_{n-1}T + a_n$$
be the minimal polynomial of $g_0$ over $A$. Since $p\nmid g$, we have $p\nmid g_0$ and thus $p^j\nmid a_j$ for some $j$. Let $p^c$ be the highest power of $p$ that divides $a_j$. There exists $m$ such that 
{ $$a_j\in \left(p^c(x_2,\dots,x_d)^m + p^{c+1}A\right) \backslash \left(p^c(x_2,\dots,x_d)^{m+1}+p^{c+1}A\right).$$}
% above needed adjustment - do we want parentheses around second set?
%

We now define a valuation $v_A$ on $A$ by setting $v_A(p)=1$ and $v_A(x_i)={1}/{2m}$ for all $i$. It follows that $v_A(a_j)=c+1/2<j$. 
%Suppose $g_0$ has conjugates $b_1,\dots,b_n$ in the Galois closure $S$ of $A[g_0]$ in $R^+$.
%
{Suppose $b_1,\dots,b_n$ is the full set of conjugates of $g_0$ in the Galois closure $S$ of $A[g_0]$ in $R^+$.}
We know that $a_j=(-1)^{j}e_j(b_1,\dots,b_n)$ where $e_j$ is the $j$-th elementary symmetric polynomial in $n$-variables. It follows that, after reordering $b_1,\dots,b_n$ and extending $v_A$ from $A$ to $S$, we may assume that $\sum_{i=1}^{j}v_S(b_j)\leq c+1/2 <j$. Thus there exists $i$ such that $v_S(b_i)<1$. Since $b_i$ is a conjugate of $g_0$, there exists (a possibly different) valuation $v$ on $S$ extending $v_A$ such that $v(g_0)<1$. Hence for any extension of $v$ to $\widehat{R^+}$, we have $v(g){=}v(g_0+ph)= v(g_0)<1$. 
\end{proof}

\section{Perfectoid rings and mixed characteristic closure operations}

In this section, we first recall some basic facts about perfectoid rings, and then we collect some results about closure operations induced by perfectoid big Cohen-Macaulay algebras and connections with extended plus closure. We start with the definition of perfectoid rings following \cite{BhattIyengarMaRegularRingsPerfectoidAlgebras} (which is equivalent to the definition in \cite{BhattMorrowScholzeIntegralPadicHodge}). 

\begin{definition}
A ring $S$ is {\it perfectoid} if it satisfies the following:
\begin{itemize}
    \item $S$ is $p$-adically complete.
    \item There exists an element $\varpi\in S$ so that $p=u\varpi^p$ for some unit $u\in S$.
    \item The Frobenius map on $S/p$ is surjective.
    \item The kernel of the Fontaine's map $\theta$: $W(S^\flat)\to S$ is principal,\footnote{We refer to \cite[Section 3]{BhattMorrowScholzeIntegralPadicHodge} for a detailed definition of $\theta$: essentially, this is the unique map lifting the natural surjection $S^\flat\to S/p$.} where $S^\flat := \varprojlim_F S/p$.
\end{itemize}
\end{definition}

In our context, $S$ will always be $p$-torsion free and $p\in S$ admits a $p$-power root $\varpi=p^{1/p}$. In this case, $S$ is perfectoid if and only if it is $p$-adically complete and the Frobenius map $S/\varpi\to S/p$ is bijective, see \cite[Lemma 3.10]{BhattMorrowScholzeIntegralPadicHodge}. 

Let $(R,\m)$ be a Noetherian complete local domain of mixed characteristic $(0,p)$ throughout the rest of this section. An $R$-algebra $B$ is called (balanced) {\it big Cohen-Macaulay} if every system of parameters of $R$ is a regular sequence on $B$ and $B\neq \m B$. The existence of such algebras in mixed characteristic has been a long standing open question and was eventually solved by Andr\'{e} \cite{AndreDirectSummandConjecture}, who later proved that there exist perfectoid balanced big Cohen-Macaulay $R^+$-algebras \cite{AndreWeakFuncorialBCM}. Bhatt proved the following remarkable result that provides an explicit construction of such algebras.

\begin{theorem}[{\cite[Corollary 5.11]{BhattCMnessAbsotluteIntegralClosure}}]
\label{theorem: BB}
%Let $(R,\m)$ be a Noetherian complete local domain of mixed characteristic $(0,p)$. Then 
$\widehat{R^+}$ is a (perfectoid) balanced big Cohen-Macaulay $R$-algebra.
\end{theorem}

We next recall that for any integral domain $S$ with $p$ in its Jacobson radical (e.g., $S$ is $p$-adically complete) and an ideal $I\subseteq S$, the (full) {\it extended plus closure} of $I$ is the collection of all elements $z\in S$ such that there exists a nonzero $c\in S$ so that $c^{1/n}z\in (I, p^n) S^+$ for all $n\in\mathbb{N}$, see \cite{HeitmannExtensionsPlusClosure}.  We will need the following result relating the extended plus closure and closures induced by perfectoid big Cohen-Macaulay algebras, this is essentially contained in \cite[Proof of Proposition 5.2.5]{CaiLeeMaSchwedeTuckerPerfdSignature}. 

%
%
%{\color{red} I do not see how the next lemma follows from the references.}
\begin{lemma}
\label{lemma: epf and BCM closure}
Let $I\subseteq R$ be an ideal that contains a power of $p$. Then $I^{epf}=IB\cap R$ for all sufficiently large perfectoid big Cohen-Macaulay $R^+$-algebras $B$. 
\end{lemma}
\begin{proof}
If $x\in I^{epf}$, then as $I$ contains a power of $p$, there exists a nonzero $c\in R$ so that $c^{1/n}x\in IR^+$ for all $n\in\mathbb{N}$. By \cite[Corollary 2.5.3]{CaiLeeMaSchwedeTuckerPerfdSignature}, there exists a perfectoid big Cohen-Macaulay $R^+$-algebra $B(x)$ so that $x\in IB(x)$. Now apply \cite[Theorem 4.9]{MaSchwedeSingularitiesMixedCharBCM} to find a sufficiently large perfectoid big Cohen-Macaulay $R^+$-algebra $B$ dominating each such $B(x)$, we obtain that $I^{epf}\subseteq IB\cap R$. 

Conversely, suppose $x\in IB\cap R$. Then by \cite[Lemma 5.1.6, $M_2=M_4$]{CaiLeeMaSchwedeTuckerPerfdSignature} applied to the finitely generated module $R/I$, we find that $g^{1/p^e} x \in IR^+$ for all $e\in\mathbb{N}$ (where $g$ is a nonzero discriminant element coming from a Noether-Cohen normalization $A\to R$). This clearly implies that $x\in I^{epf}$.  
\end{proof}
To prove the next lemma, we follow some notations as in \cite{CaiLeeMaSchwedeTuckerPerfdSignature} (see also \cite{BMPSTWWPerfdPure}). Let $C_k$ be a coefficient ring of $R$ (where $k$ denotes the residue field of $R$) and fix an inclusion $C_k\hookrightarrow W(k^{1/p^\infty})$, where the latter is the ring of Witt vectors of the perfect field $k^{1/p^\infty}$. Let $A:= C_k[[x_2,\dots, x_d]]\hookrightarrow R$ be any Noether-Cohen normalization of $R$. We set $A_\infty$ to be the $p$-adic completion of
$$(A\widehat{\otimes}_{C_k}W(k^{1/p^\infty}))[p^{1/p^\infty}, x_2^{1/p^\infty},\dots, x_d^{1/p^\infty}].$$ 
%$$W(k^{1/p^\infty})[[x_1,\dots,x_n]][p^{1/p^\infty}, x_1^{1/p^\infty},\dots, x_n^{1/p^\infty}].$$ 
Then $A_\infty$ is a perfectoid ring and we set
$R^{A_\infty}_{\perfd}:=(A_\infty\otimes_AR)_{\perfd}$ to 
be the perfectoidization of $A_\infty\otimes_AR$ in the sense of Bhatt--Scholze \cite[Sections 7 and 8]{BhattScholzepPrismaticCohomology}. Note that since we have a (non-canonical) map $A_\infty\to \widehat{R^+}$, by the universal property of the perfectoidization functor, we have a map $R^{A_\infty}_{\perfd}\to \widehat{R^+}$.

%The almost purity theorem \cite[Theorem 10.9]{BhattScholzepPrismaticCohomology}
%implies that if $A[1/g]\to R[1/g]$ is finite \'etale, then $A_\infty\to R^{A_\infty}_{\perfd}$ is $g$-almost finite \'etale. 

\begin{lemma}
\label{lemma: p is test element}
Suppose $R[1/p]$ is regular. Then for every ideal $I\subseteq R$ that contains a power of $p$, we have $(p^{1/p^\infty})I^{epf}\subseteq IR^+$. 
\end{lemma}
\begin{proof}
Let $J$ be the ideal generated by all elements $g\in R$ such that there exists some Noether-Cohen normalization $A\hookrightarrow R$ with $g\in A$ so that $A[1/g]\to R[1/g]$ is finite \'etale. We first show that $p\in \sqrt{J}$. If not, then we can find a prime $Q\in\Spec(R)$ containing $J$ such that $p\notin Q$. Since $R[1/p]$ is regular, we know that $R_Q$ is regular. But then by \cite[Theorem 0.1]{HeitmannEtaleLocus}, there exists a Noether-Cohen normalization $A\to R$ and $g\in A-Q$ such that $A[1/g]\to R[1/g]$ is finite \'etale contradicting our choice of $g$. It follows that there exist Noether-Cohen normalizations $A_i\to R$, $1\leq i\leq n$, such that $A_i[1/g_i]\to R[1/g_i]$ is finite \'etale and $p\in \sqrt{(g_1,\dots,g_n)}$. 

% SSP
%
Now suppose $z\in I^{epf}$.  Then by Lemma~\ref{lemma: epf and BCM closure}, we have that $z\in IB\cap R$ for all sufficiently large perfectoid big Cohen-Macaulay $R^+$-algebras $B$ and thus by \cite[Lemma 5.1.6, $M_1=M_4$]{CaiLeeMaSchwedeTuckerPerfdSignature}, we know that
%\footnote{In \cite{CaiLeeMaSchwedeTuckerPerfdSignature} it is assumed that $k$ is perfect but the same argument works in our setting, here we recall that for a perfectoid ring $S$ and a finitely generated ideal $J\subseteq S$, $J_{\perfd}$ is the kernel of the canonical map $S\to (S/J)_{\perfd}$, see \cite[Section 10]{BhattScholzepPrismaticCohomology}.} 
$(g_i)_{\perfd}z\in IR^{(A_i)_\infty}_{\perfd}$
for every $i$ (for a perfectoid ring $S$ and a finitely generated ideal $J\subseteq S$, $J_{\perfd}$ is the kernel of the natural map $S\to (S/J)_{\perfd}$, see \cite[Section 10]{BhattScholzepPrismaticCohomology}). Since $R^{(A_i)_\infty}_{\perfd}$ maps to $\widehat{R^+}$ for every $i$, it follows that 
$$\left((g_1)_{\perfd} + \cdots + (g_n)_{\perfd}\right)z = (g_1,\dots, g_n)_{\perfd}z \in I\widehat{R^+}$$
where the equality follows from \cite[Proposition 8.13]{BhattScholzepPrismaticCohomology} (a priori it guarantees the two ideals $(g_1,\dots, g_n)_{\perfd}$ and $(g_1)_{\perfd} + \cdots + (g_n)_{\perfd}$ agree up to $p$-adic closure, but both ideals are $p$-adically closed, see \cite[Proposition 2.7]{FayolleCentersPerfdPurity}). But since $p\in \sqrt{(g_1,\dots,g_n)}$ and the ideal $(g_1,\dots,g_n)_{\perfd}$ is radical, we have that $(p^{1/p^\infty})\subseteq (g_1,\dots,g_n)_{\perfd}$ and thus 
$$(p^{1/p^\infty})z\in I\widehat{R^+}\cap R^+=IR^+$$
since $I$ contains a power of $p$. This completes the proof of the lemma.
\end{proof}

Now we are ready to state and prove the key lemma that will be used in the next section. 

\begin{lemma}
\label{lemma: epf is p almost in ideal}
Suppose $x,y$ is a regular sequence in $R^+$ and that $(x,y)R^+$ contains a power of $p$. Then 
$$(p^{1/p^\infty})((x,y)R^+)^{epf}\subseteq (x,y)R^+.$$
\end{lemma}
\begin{proof}
%SSP
%
We first assume $\dim(R)=2$ and let $z\in ((x,y)R^+)^{epf}$.  We may assume that $x,y,z\in S$ where $S$ is a complete normal local domain module-finite over $R$. Since $\dim(S)=2$ and $S$ is normal, $S[1/p]$ is regular and thus by Lemma~\ref{lemma: p is test element}, $(p^{1/p^\infty})z\in (x,y)R^+$ as wanted. 

We now suppose $R$ is a counter-example to the lemma of smallest dimension, and by the previous paragraph we have $\dim(R)\geq 3$. 
\begin{claim}
\label{claim: support only at m}
The $R^+$-module
\[\frac{(p^{1/p^\infty})((x,y)R^+)^{epf}+(x,y)R^+}{(x,y)R^+}\]
is supported only possibly at $\m_{R^+}$.
\end{claim} 
\begin{proof}[Proof of Claim]
Suppose $Q\in\Spec(R^+)$ is in the support and $Q\neq \m_{R^+}$ (note that $Q$ necessarily contains $p$). Then 
\[\frac{(p^{1/p^\infty})((x,y)R^+)^{epf}R^+_Q+(x,y)R^+_Q}{(x,y)R^+_Q}\neq 0.\]

Thus there exists $m$ and $z\in p^{1/p^m}((x,y)R^+)^{epf}$ such that $z\notin (x,y)R^+_Q$. We may assume that $x,y,z\in S$ where $S$ is a complete normal local domain module-finite over $R$. Let $P$ be the contraction of $Q$ to $S$ and let $T$ be the $(PS_P)$-adic completion of $S_P$. Since $S_P$ is excellent and normal, $T$ is a Noetherian complete 
%
%
%{\color{red} I don't see how this is given by the Datta-Tucker reference, which should be 3.2.5 in any case.}
local domain of mixed characteristic $(0,p)$ such that $\dim(T)<\dim(R)$. Abusing {notation}, we still use $x,y,z$ to denote their images in various localizations and completions. Since $z\notin (x,y)R^+_Q$, we have $z\notin (x,y)(R^+)_P=(x,y)(S_P)^+$. This implies that $z\notin (x,y)T^+$ by \cite[Proposition 3.10]{DattaTuckerPermanencePropertiesSplinters} (which originates from \cite[Proposition 5.10]{Smithtightclosureofparameterideals}). But clearly, $z\in p^{1/p^m}((x,y)T^+)^{epf}\subseteq (p^{1/p^\infty})((x,y)T^+)^{epf}$ and thus $T$ is a counter-example to the lemma with $\dim(T)<\dim(R)$, a contradiction. 
\end{proof}

By Claim~\ref{claim: support only at m}, we have
\begin{align*}
  \frac{(p^{1/p^\infty})((x,y)R^+)^{epf}+(x,y)R^+}{(x,y)R^+}  
    & \cong  H_\m^0\left(\frac{(p^{1/p^\infty})((x,y)R^+)^{epf}+(x,y)R^+}{(x,y)R^+}\right) \\
    & \hookrightarrow H_\m^0\left(\frac{R^+}{(x,y)R^+}\right) \cong H_\m^0\left(\frac{\widehat{R^+}}{(x,y)\widehat{R^+}}\right)=0 
\end{align*}
where the isomorphism on the second line follows from the fact that ${(x,y)R^+}$ contains a power of $p$ and the last equality on the second line follows from Theorem~\ref{theorem: BB} (and that $\dim(R)\geq 3$). This completes the proof of the lemma.
\end{proof}

\section{The main results}

\noindent\textbf{Notation}: Throughout this section, $(R,\m)$ will denote a Noetherian complete local domain of mixed characteristic $(0,p)$. We fix an absolute integral closure $R^+$ of $R$ and let $\widehat{R^+}$ be its $p$-adic completion. We will write an element $z\in\widehat{R^+}$ as a limit 
\[z=\lim_i z_i, \text{ where } z_i\in R^+ \text{ and } z_{i} \equiv  z_{i-1} \text{ mod } p^i  \text{ for all $i$}.\]
Moreover, for every height one prime $Q\in\Spec(R^+)$ that contains $p$, $(R^+)_Q$ is a valuation ring of rank one and thus defines a $\mathbb{Q}$-valuation $v_Q$ on $R^+$. We normalize $v_Q$ so that $v_Q(p)=1$, and we extend it to an $\mathbb{R}$-valuation on $\widehat{R^+}$ by assigning the value of $z\in\widehat{R^+}$ to be $\lim_iv_Q(z_i)$ (it is easy to check that this is well-defined). We will abuse {notation} a bit and still write $v_Q$ for the extended valuation on $\widehat{R^+}$. 

We caution the readers that, even though $\widehat{R^+}$ is an integral domain, the valuation $v_Q$ has nontrivial support on $\widehat{R^+}$ as long as $\dim(R)\geq 2$, see Proposition~\ref{proposition: construction of bad g}.

\vspace{0.7em}

We are ready to state and prove the following extension of Theorem B from the introduction. 

\begin{theorem}
\label{theorem: completeIntegrallyClosed}
With notation as above. For a nonzero $g\in \widehat{R^+}$, the following are equivalent:
\begin{enumerate}
  \item[(1)] $\widehat{R^+}/g$ is $p$-adically separated.
  \item[(1')] $\widehat{R^+}/g$ is $p$-adically complete.
  \item[(2)] $v_Q(g)<\infty$ for all height one primes $Q\in\Spec(R^+)$ that contain $p$.
  \item[(2')] $\exists N>0$ such that $v_Q(g)<N$ for all height one primes $Q\in\Spec(R^+)$ that contain $p$.
  \item[(3)] $\widehat{R^+}/g$ has bounded $p$-power torsion.
  \item[(4)] $\widehat{R^+}$ is completely integrally closed in $\widehat{R^+}[1/g]$.
\end{enumerate}
\end{theorem}

\begin{remark}
\label{remark: derived completion}
We first point out that $(3)\Rightarrow(1')\Leftrightarrow(1)$ follows from general facts on derived completion. Since $\widehat{R^+}$ is $p$-adically complete, we know that $\widehat{R^+}/g$ is derived $p$-complete by \cite[\href{https://stacks.math.columbia.edu/tag/091U}{Tag 091U}]{stacks-project} and thus it is $p$-adically complete if and only if it is $p$-adically separated by \cite[\href{https://stacks.math.columbia.edu/tag/091T}{Tag 091T}]{stacks-project}. Moreover, in the bounded $p$-power torsion case (e.g., $p$-torsion free case), derived $p$-completion agrees with usual $p$-adic completion as mentioned before, thus $\widehat{R^+}/g$ is $p$-adically complete. 
\end{remark}

The proof of Theorem~\ref{theorem: completeIntegrallyClosed} consists of many steps and we need several lemmas. We start by showing that the conditions $(2)$ and $(2')$ above are equivalent. 

\begin{lemma}
\label{lemma: v uniform bounded}
With notation as in Theorem~\ref{theorem: completeIntegrallyClosed}, we have $(2)\Leftrightarrow(2')$.
\end{lemma}
\begin{proof}
It suffices to show $(2)\Rightarrow(2')$ as the other direction is obvious. 
We write $g=\lim_ng_n$ and let $R_n$ be the normalization of $R[g_0, g_1,\dots,g_n]$. Note that we have 
$$R_0\subseteq R_1\subseteq R_2\subseteq \cdots \subseteq R^+.$$ 

Suppose $v_Q(g)$ is not uniformly bounded. Then for every $n$, there exists $Q$ (which might depend on $n$) such that $v_Q(g_n)\geq n+1$, since otherwise $v_Q(g)=v_Q(g_n)<n+1$ for all $Q$ (as $g\equiv g_n$ mod $p^{n+1}$). We next note that if $v_Q(g_n)\geq n+1$, then $v_Q(g_m)\geq m+1$ for all $m\leq n$: since otherwise $v_Q(g_n)=v_Q(g_m)<m+1\leq n+1$ (as $g_n\equiv g_m$ mod $p^{m+1}$).

Now let $\Lambda_n$ be the set of height one primes $P$ in $R_n$ such that $P$ is the contraction of a height one prime $Q$ of $R^+$ containing $p$ with $v_Q(g_n)\geq n+1$ (note that since $g_n\in R_n$, this implies that $v_{Q'}(g_n)\geq n+1$ {for every height one prime $Q'$} of $R^+$ that contracts to $P$). By the discussion above, we know that $\Lambda_n\neq\emptyset$ for all $n$ and that the contraction of each $P\in \Lambda_n$ to $R_m$ is contained in $\Lambda_m$ for all $m\leq n$. For each $n'\geq n$, we further define
$$\Theta_n^{n'} := \{P\in \Lambda_n \,\ | \,\ \text{$P$ is the contraction of some $P'\in \Lambda_{n'}$} \}.$$
Again by the discussion before, $\{\Theta_n^{n'}\}_{n'}$ form a descending chain of non-empty subsets of $\Lambda_n$. Since $\Lambda_n$ is a finite set (each $R_n$ is Noetherian and each prime ideal in $\Lambda_n$ is a minimal prime of $p$), $\{\Theta_n^{n'}\}_{n'}$ stabilizes to a nonempty subset $\Theta_n$ of $\Lambda_n$. Moreover, it is easy to see that for every $n$ and every $P\in \Theta_n$, there exists $P'\in \Theta_{n+1}$ such that $P'$ contracts to $P$. Therefore we can start with a prime $P_0\in \Theta_0$ in $R_0$ and compatibly lift $P_0$ to a chain 
$$P_0\subseteq P_1 \subseteq P_2 \subseteq \cdots, $$
where $P_i\in \Theta_i$ is a prime in $R_i$ that contracts to $P_{i-1}$ for each $i$. It follows that there exists a prime $Q\in\Spec(R^+)$ that contracts to $P_i$ for each $i$. However, since $P_i\in \Theta_i\subseteq \Lambda_i$, we know that $v_Q(g_i)\geq i+1$ for each $i$ and thus $v_Q(g)=\lim_iv_Q(g_i)=\infty$, a contradiction. 
\end{proof}

%typo fixed
%
The next two lemmas will establish $(2)\Rightarrow(3)$ and $(4)\Rightarrow(1)$. 

\begin{lemma}
\label{lemma: bounded p-torsion}
With notation as in Theorem~\ref{theorem: completeIntegrallyClosed}, we have $(2)\Rightarrow(3)$.
\end{lemma}
\begin{proof}
Assuming $(2)$ holds, by Lemma~\ref{lemma: v uniform bounded}, there exists $N>0$ such that $v_Q(g)<N$ for all $Q$. It is enough to show that any $p^{N+1}$-torsion in $\widehat{R^+}/g$ is $p^N$-torsion. Suppose $p^{N+1}h\in g\widehat{R^+}$, i.e., we have $p^{N+1}h=gz$ for some $z\in \widehat{R^+}$. Then $v_Q(z)>1$ for every height one prime $Q\in\Spec(R^+)$ containing $p$. Writing $z=z_0+pw$ with $z_0\in R^+$, we have $v_Q(z_0)\geq 1$ and thus $z_0\in\overline{pR^+}=pR^+$ (see \cite[Proposition 1.5.2]{SwansonHunekeIntegralClosure} for the equality). Therefore $z\in p\widehat{R^+}$ and thus $p^N h\in g\widehat{R^+}$.
\end{proof}

\begin{lemma}
\label{lemma: p adic closure in total integral closure}
Let $S$ be a domain and let $J\subseteq S$ be an ideal such that $S$ is $J$-adically complete. Suppose $f,g\in S$ so that $f$ is in the $J$-adic closure of $(g)$, i.e., $f\in \cap_n ((g) + J^n)$. Then $f/g\in S[1/g]$ is almost integral over $S$. 
In particular, with notation as in Theorem~\ref{theorem: completeIntegrallyClosed}, we have $(4)\Rightarrow(1)$.
\end{lemma}
\begin{proof}
Write $f= \lim_i f_i$ where $f_i\equiv  f_{i+1}$ mod $J^{i+1}$ for all $i$, and similarly write $g=\lim_i g_i$. By assumption, we know that $f_i\equiv a_ig_i$ mod $J^{i+1}$ for some $a_i\in S$. We may replace $f_i$ by $a_ig_i$ without changing $f$ thus we may assume that $f_i=a_ig_i$ for every $i$. Now the condition $f_{n+1}-f_n\in J^{n+1}$ translates to $$a_{n+1}g_{n+1}-a_ng_n=a_{n+1}(g_{n+1}-g_n)+g_n(a_{n+1}-a_n)\in J^{n+1},$$ which is equivalent to $g_n(a_{n+1}-a_n)\in J^{n+1}$ since $g_{n+1}-g_n\in J^{n+1}$. Thus we have 
$$a_{n+1}-a_n\in (J^{n+1}:g_n).$$
Let $k\geq 2$ be an integer and let $h_n=a_n^kg_n$.
Since $a_{n+1}-a_n$ divides $a_{n+1}^k-a_n^k$, we have $a_{n+1}^k-a_n^k\in (J^{n+1}:g_n)$ and so $h_{n+1}-h_n\in J^{n+1}$. Hence the sequence $\{h_n\}_{n}$ defines an element $h\in S$.
By construction we see that $g^{k-1}h=f^k$. It follows that $g(f/g)^k\in S$ for all $k\in\mathbb{N}$, which implies that $f/g$ is almost integral over $S$.

The last conclusion follows by applying the first conclusion to $S=\widehat{R^+}$ and $J=(p)$. 
\end{proof}

%Suppose $f$ is in the $p$-adic closure of $g\widehat{R^+}$, that is, $f\in \cap_n(g, p^n)\widehat{R^+}$. Then $f/g\in \widehat{R^+}[1/g]$ is almost integral over $\widehat{R^+}$, i.e., $f/g$ is in the complete integral closure of {$\widehat{R^+}$} in $\widehat{R^+}[1/g]$. In particular, we have $(4)\Rightarrow(1)$.
%Write $f=\lim_i f_i$ and $g=\lim_i g_i$; we have that $f_i=a_ig_i+p^{i+1}b_i$ where $a_i, b_i\in R^+$. Replacing $f_i$ by $f_i-p^{i+1}b_i$, we may assume that $f_i=a_ig_i$ for every $i$. Now the condition $f_{n+1}-f_n\in p^{n+1}R^+$ translates to $$a_{n+1}g_{n+1}-a_ng_n=a_{n+1}(g_{n+1}-g_n)+g_n(a_{n+1}-a_n)\in p^{n+1}R^+,$$ which is equivalent to $g_n(a_{n+1}-a_n)\in p^{n+1}R^+$ since $g_{n+1}-g_n\in p^{n+1}R^+$. Thus we have 
%$$a_{n+1}-a_n\in (p^{n+1}:_{R^+}g_n).$$
%Let $k\geq 2$ be an integer and let $h_n=a_n^kg_n$.
%Since $a_{n+1}-a_n$ divides $a_{n+1}^k-a_n^k$, we have $a_{n+1}^k-a_n^k\in (p^{n+1}:g_n)$ and so $h_{n+1}-h_n\in p^{n+1}R^+$. Hence the sequence $\{h_n\}_{n}$ defines an element $h\in \widehat{R^+}$.
%By construction we see that $g^{k-1}h=f^k$. It follows that $g(f/g)^k\in \widehat{R^+}$ for all $k\in\mathbb{N}$.
%Thus $f/g$ is almost integral over $\widehat{R^+}$.

We next prove $(1)\Rightarrow(2)$. This proof will utilize the full result in \cite{HeitmannAICSurprisinglyDomain} that the $\m$-adic completion of $R^+$ is an integral domain.

\begin{lemma}
\label{lemma: HeitmannNote}
With notation as in Theorem~\ref{theorem: completeIntegrallyClosed}, we have $(1)\Rightarrow(2)$.
\end{lemma}
\begin{proof}
First note that if $\dim(R)=1$, then $\widehat{R^+}$ is a rank one valuation ring, and $(1)$ and $(2)$ hold for all nonzero $g\in \widehat{R^+}$ (the only height one prime of $R^+$ is $\m_{R^+}$). Thus in what follows we assume $\dim(R)\geq 2$.

Assuming $(2)$ is not true, we will construct an element $f$ in the $p$-adic closure of $g\widehat{R^+}$. We write $g=\lim_n g_n$. Then, following the proof of Lemma~\ref{lemma: p adic closure in total integral closure} (in the case $S=\widehat{R^+}$ and $J=(p)$), it is enough to find a sequence $\{a_n\}_n$ of elements in $R^+$ with $a_{n+1}-a_n\in (p^{n+1}:g_n)$. Because then we can take $f_n=a_ng_n$ and $f=\lim_n f_n$ will be in the $p$-adic closure of $g\widehat{R^+}$. 

We will choose our sequence $\{a_n\}_n$ via a recursive procedure. 
Suppose there exists $v_Q$ such that $v_Q(g)=\infty$.  We write $g=\lim_ig_i$ and we can assume that the sequence $g_i$ is chosen so that $v_Q(g_i)=i+1$ for all $i\in\mathbb{N}$. 
Let $v$ be a rank one valuation on $R^+/Q$ which extends a discrete valuation on $R/(Q\cap R)$. Pick $x\in R$ such that $p,x$ is part of a system of parameters of $R$ and that $v(x)=1$.
First we choose $a_0\notin Q$ and let $\ell: =v(a_0)$. Fix an $m_0\in\mathbb{N}$ so that $m_0>\ell$. 
Now suppose $a_n$ and $m_n$ have been chosen with $m_n>\ell$.
Since $v_Q(g_n)=n+1$ for each $n$, we have $(p^{n+1}:g_n)\nsubseteq Q$.
Thus we may choose $c_{n+1}'\in (p^{n+1}:_{R^+}g_n)-Q$.
Let $\ell_n=v(c_{n+1}'x^{m_n})= v(c'_{n+1})+m_{n}\geq m_n$.
We may harmlessly multiply $c'_{n+1}$ by a small power of $x$ to assume that the subgroup of $\mathbb{Q}$ generated by $\ell$ and $\ell_n$ is $\langle1/j\rangle$ with $j>n^2$.
Choose $m_{n+1}\in\mathbb{N}$ so that $m_{n+1}>\ell_n$, and set $c_{n+1}''=c_{n+1}'x$.
Let $a_{n+1}'=a_n+c_{n+1}'x^{m_n}$ and $a_{n+1}''=a_n+c_{n+1}''x^{m_n}$.
We claim that we cannot simultaneously find elements $b_n',b_n''\in R^+$ which satisfy monic polynomials of degree $n$ and which are congruent modulo $x^{m_{n+1}}$ to $a_{n+1}'$ and $a_{n+1}''$ respectively.
For otherwise there is an integral extension of $R$ of degree at most $n^2$ which contains $b_n',b_n''$.
Since $v(b_n')=v(a_n')=\ell$ and 
$$v(b_n'-b_n'')=v(c'_{n+1}x^{m_n}-c''_{n+1}x^{m_n}) =v(c'_{n+1}x^{m_n})=\ell_n,$$ 
the ramification of the extension is at least $j>n^2$, a contradiction.
This allows us to choose $a_{n+1}$ so that it is either $a_{n+1}'$ or $a_{n+1}''$ and cannot be lifted to an element $b_n$ which satisfies a monic polynomial over $R$ of degree $n$.
In either case, $v(a_{n+1})=v(a_n)=\ell$. 

We claim that $f\notin g\widehat{R^+}$. To see this, suppose on the contrary that $f=bg$ for some $b\in \widehat{R^+}$.
Since the $\m$-adic completion of $R^+$ is an integral domain by \cite[Theorem 1.5]{HeitmannAICSurprisinglyDomain}, we must have $b=a$ where $a$ is the limit of the sequence $\{a_n\}_n$ in the $\m$-adic completion of $R^+$ (note that $a$ is well-defined since $a_{n+1}-a_n\in (x^{m_n})R^+\subseteq \m^{m_n}R^+$ and the sequence $\{m_n\}_n$ is increasing).
Thus $a\equiv b$ mod $Q$ and so $\overline{a}$, the image of $a$ in $R^+/Q$, is the image of an element in $R^+$.
In particular, $\overline{a}$ is algebraic over the fraction field of ${R}/(Q\cap R)$. 
In that case, it must satisfy a monic polynomial of degree $j$ for some $j$.
However, no such element can be congruent to $a_{j+1}$ modulo $(p,x^{m_{j+2}})$ and we have our desired contradiction.
\end{proof}

Putting Remark~\ref{remark: derived completion}, Lemma~\ref{lemma: v uniform bounded}, Lemma~\ref{lemma: bounded p-torsion}, Lemma~\ref{lemma: p adic closure in total
%SSP
%
 integral closure}, and Lemma~\ref{lemma: HeitmannNote} together, we have proved that $$(1)\Leftrightarrow(1')\Leftrightarrow(2)\Leftrightarrow(2')\Leftrightarrow(3) \text{ and } (4)\Rightarrow(1)$$ in Theorem~\ref{theorem: completeIntegrallyClosed}. It remains to show that the already established equivalent conditions $(1)-(3)$ imply $(4)$. To prove this final
 %SSP
%
implication, we need to use various results on mixed characteristic closure operations in Section 3 and the following version of the Brian\c{c}on-Skoda theorem for $R^+$.

\begin{theorem}[{\cite[Corollary 5.1]{RodriguezVillalobosSchwedeBrianconSkoda}}, see also {\cite[Theorem 2.13]{HeitmannPlusClosureMixedChar}}]
\label{theorem: BrianconSkoda}
Let $I\subseteq R^+$ be an ideal generated by $h$ elements such that $p\in \sqrt{I}$. Then $\overline{I^{k+h-1}}\subseteq I^{k}$ for every $k\geq 0$. In particular, if $I$ is generated by two elements and $p\in\sqrt{I}$, then $\overline{I^{k}}\subseteq I^{k-1}$ for every $k>0$. 
\end{theorem}

We need one more ingredient.

\begin{lemma}
\label{lemma: regular sequence}
Suppose $g$ is a nonzero element in $\widehat{R^+}$ such that $\widehat{R^+}/g$ is $p$-adically separated. Then there exists an $N>0$ and a nonzero element $g'\in \widehat{R^+}$ such that $p^Ng'\in g\widehat{R^+}$ and $p, g'$ form a regular sequence in $\widehat{R^+}$. 
\end{lemma}
\begin{proof}
By Lemma~\ref{lemma: HeitmannNote} and Lemma~\ref{lemma: v uniform bounded}, there exists $N>0$ such that $v_Q(g)<N$ { for all height one primes $Q$ of $R^+$ which contain $p$.} Since for each {such} $Q\in\Spec(R^+)$, $(R^+)_Q$ is a valuation ring of rank one, we have 
$$I:= (g:_{\widehat{R^+}}p^N)(\widehat{R^+}/p) \nsubseteq Q(\widehat{R^+}/p)=Q(R^+/p).$$
It follows that $I$ has height at least one in $R^+/p$. This implies that there exists $\overline{g'}\in I$ that is a nonzerodivisor on $R^+/p$. 
{Take any preimage of $\overline{g'}$ in $\widehat{R^+}$ and we find} that there exists $g'\in (g:_{\widehat{R^+}}p^N)$ such that $p, g'$ is a regular sequence on $\widehat{R^+}$ as wanted.
\end{proof}

\begin{proof}[Proof of Theorem~\ref{theorem: completeIntegrallyClosed}]
As noted before, it remains to prove $(1)\Rightarrow(4)$. We first prove that $\widehat{R^+}$ is completely integrally closed in $\widehat{R^+}[1/p]$. Suppose ${r}/{p^m}$ is almost integral over $\widehat{R^+}$. Writing $r=s+p^mt$ where $s\in R^+$ and $t\in \widehat{R^+}$, we have ${r}/{p^m}={s}/{p^m}+t$ and thus $z:={s}/{p^m}$ is almost integral over 
%SSP
%
$\widehat{R^+}$.  It suffices to show that $z\in R^+$. By the definition of almost integral elements, we know that there exists $c\in\widehat{R^+}$ so that $cs^n\in p^{mn}\widehat{R^+}$ for all $n$. Writing $c=\lim_j c_j$ we know that 
$$c_{mn}s^n\in p^{mn}\widehat{R^+}\cap R^+=p^{mn}R^+$$ for all $n$. It follows that 
$$c_{mn}^{1/n}s\in \overline{p^mR^+}=p^mR^+.$$
By Lemma~\ref{lemma: v(g) is not infinity}, there exists an $\mathbb{R}$-valuation $v$ centered on $\m_{\widehat{R^+}}$ such that $v(c)=\lambda<\infty$. It follows that $v(c_{mn})<\lambda+1$ for all $n$ sufficiently large and thus $v(c_{mn}^{1/n})\to 0$ as $n\to\infty$. By \cite[Corollary 2.5.3]{CaiLeeMaSchwedeTuckerPerfdSignature}, there exists a
%SSP
%
 perfectoid big Cohen-Macaulay $R^+$-algebra $B$ such that $s\in p^mB\cap R^+$. But then by Lemma~\ref{lemma: epf and BCM closure} (enlarging $R$ if necessary so that we may assume $s\in R$), we know that 
$$s\in (p^m)^{epf}\subseteq \overline{p^mR^+}=p^mR^+.$$
%SSP
%
Thus $z={s}/{p^m}\in R^+$ as wanted.

Next, we show that $\widehat{R^+}$ is completely integrally closed in $\widehat{R^+}[1/g]$ when $p,g$ is a regular sequence in $\widehat{R^+}$. Note that this implies $g,p$ is a regular sequence on $\widehat{R^+}$ as $\widehat{R^+}$ is an integral domain (and in particular, $\widehat{R^+}/g$ is $p$-adically separated since $\widehat{R^+}/g$ is $p$-torsion free and derived $p$-complete, see Remark~\ref{remark: derived completion}). Suppose $f/g^e$ is almost integral over $\widehat{R^+}$, replacing $g$ by $g^e$ we may assume $e=1$. Thus we assume there exists $h\in \widehat{R^+}$ so that $hf^m\in g^m\widehat{R^+}$ for all $m>0$, and we need to show $f\in g\widehat{R^+}$. Writing $f=\lim_if_i$, $g=\lim_ig_i$, and 
%SSP
%
$h=\lim_ih_i$ for all $m,j$, we have
$$h_jf_j^m\in(g_j^m, p^{j+1})\widehat{R^+}\cap R^+=(g_j^m,p^{j+1})R^+\subseteq (g_j^m,p^{j})R^+.$$ 
It follows that for all $m, j, k \in \mathbb{N}$, we have
%SSP
%
$$h_j^{1/m}f_j\in \overline{(g_j, p^{j/m})R^+} \subseteq \overline{(g_j^{1/k}, p^{j/mk})^kR^+}\subseteq (g_j^{1/k},p^{j/mk})^{k-1}R^+\subseteq (g_j^{1-1/k},p^{j/mk})R^+ ,$$
where we have used Theorem~\ref{theorem: BrianconSkoda} for the second $\subseteq$. Multiplying by $g_j^{1/k}$, we obtain that 
$$g_j^{1/k}h_j^{1/m}f_j \in (g_j, p^{j/mk})R^+ \text{ for all $m, j, k$.}$$
%SSP
%
Now we take $j=mkn$ and the above becomes 
$$g_{mkn}^{1/k}h_{mkn}^{1/m}f_{mkn} \in (g_{mkn}, p^{n})R^+ \text{ for all $m,k, n$.}$$
{Since $f_{mkn}\equiv f_n$ and $g_{mkn}\equiv g_n$ mod $p^n$,} we have 
$$g_{mkn}^{1/k}h_{mkn}^{1/m}f_{n} \in (g_{n}, p^{n})R^+ \text{ for all $m,k, n$.}$$
We now invoke Lemma~\ref{lemma: v(g) is not infinity} to find an $\mathbb{R}$-valuation $v$ centered on $\m_{\widehat{R^+}}$ so that $v(gh)=\lambda<\infty$. It follows that $v(g_{mkn}h_{mkn})<\lambda+1$ for all $m,k$ sufficiently large and thus $v(g_{mkn}^{1/k}h_{mkn}^{1/m})\to 0$ when $m,k \to\infty$. It follows from \cite[Corollary 2.5.3]{CaiLeeMaSchwedeTuckerPerfdSignature} that there exists a perfectoid big Cohen-Macaulay $R^+$-algebra $B$ such that $f_n\in (g_n, p^n)B\cap R^+$. By Lemma~\ref{lemma: epf
%SSP
%
and BCM closure} (again enlarging $R$ if necessary to assume $f_n, g_n \in R$) we have 
$$f_n\in (g_n, p^n)^{epf}\subseteq ((g_n, p^n)R^+)^{epf}.$$ 
Since $p, g$ is a regular sequence on $\widehat{R^+}$, we have $p^n, g_n$ is a regular sequence on $R^+$ and thus by Lemma~\ref{lemma: epf is p almost in ideal}, $(p^{1/p^\infty})f_n\in (g_n, p^n)R^+$ and in particular $pf_n\in (g_n, p^n)R^+$. It follows that $pf \in (g, p^n)\widehat{R^+}$ for all $n$ and thus 
$$pf \in \cap_n(g, p^n)\widehat{R^+}=g\widehat{R^+}$$
where the equality follows from the fact that $\widehat{R^+}/g$ is $p$-adically separated. Since $g, p$ is a regular sequence on $\widehat{R^+}$, the above implies that $f\in g\widehat{R^+}$ as wanted.

Finally, {we consider the general case and only assume} $\widehat{R^+}/g$ is $p$-adically separated. By Lemma~\ref{lemma: regular sequence}, there exists $g'\in \widehat{R^+}$ such that $p,g'$ is a regular sequence and $\widehat{R^+}[1/g]\subseteq \widehat{R^+}[1/pg']$. It is enough to show that $\widehat{R^+}$ is completely integrally closed in $\widehat{R^+}[1/pg']$. Now if $\eta:={r}/{(pg')^m} \in \widehat{R^+}[1/pg']$ is almost integral over $\widehat{R^+}$, then $p^m\eta={r}/{g'^m} \in \widehat{R^+}[1/g']$ is almost integral over $\widehat{R^+}$. By the already established case above, we know that $p^m\eta \in \widehat{R^+}$. But then $\eta\in \widehat{R^+}[1/p]$ is almost integral over $\widehat{R^+}$ and thus by the already established case above, $\eta\in \widehat{R^+}$ as wanted. 
\end{proof}

We record the following consequence of Theorem~\ref{theorem: completeIntegrallyClosed}. This is a slight extension of \cite[Proposition 3.5]{KinouchiShimomotoPerfectoidRingsGaloisCohomology} (the result was originally obtained in \cite[Proposition A.6]{AndreattaGeneralizedRingNorms}). 

\begin{corollary}
\label{corollary: A_infty}
Let $(A,\m, k)$ be a complete regular local ring of mixed characteristic $(0,p)$ and let $A_\infty$ be the perfectoid ring constructed in the paragraph before Lemma~\ref{lemma: p is test element}. Then $A_\infty$ is a domain completely integrally closed in its fraction field.
\end{corollary}
\begin{proof}
Following the first paragraph of \cite[Proof of Proposition 3.5]{KinouchiShimomotoPerfectoidRingsGaloisCohomology}, we fix an injection $A_\infty\to \widehat{A^+}$ (in particular, this implies $A_\infty$ is a domain by \cite[Proposition 1.6]{HeitmannAICSurprisinglyDomain}) and we have a $\mathbb{Z}[1/p]$-valued valuation $v$ on $A_\infty$ so that for any $x\in A_\infty$,
$$v(z)=\max\{\frac{n}{p^e} \mid z\in p^{n/p^e}A_\infty\}.$$
Thus for any $0\neq g\in A_\infty$, we can write $g=p^{n/p^e}g'$ where $n/p^e=v(g)$ and $g'\notin (p^{1/p^\infty})A_\infty$. It follows that $p, g'$ is a regular sequence on $A_\infty$ as $(p^{1/p^\infty})A_\infty$ is a prime ideal in $A_\infty$ (which is the center of $v$). 

\begin{claim}
\label{claim: p-adic separated}
For any $g\in A_\infty$, both $A_\infty/g$ and $\widehat{R^+}/g$ are $p$-adically separated.
\end{claim}
\begin{proof}
We may assume $g\neq 0$ and write $g=p^{n/p^e}g'$ as above. We have an inclusion:
$$A_\infty/g \hookrightarrow A_\infty/p^{n/p^e} \oplus A_\infty/g'.$$
Since both $A_\infty/p^{n/p^e}$ and $A_\infty/g'$ are $p$-adically separated ($A_\infty/g'$ is $p$-torsion free and derived $p$-complete and hence classically $p$-adically complete, see Remark~\ref{remark: derived completion}), so is $A_\infty/g$. Finally, to see that $\widehat{R^+}/g$ is $p$-adically separated, it is enough to note that $v_Q(g)={n}/{p^e}<\infty$ for all height one primes $Q\in\Spec(R^+)$ that contain $p$ and apply Theorem~\ref{theorem: completeIntegrallyClosed} $(2)\Rightarrow(1)$.\footnote{Alternatively, one can also use the $p$-complete flatness of $A_\infty\to\widehat{R^+}$ to see that $p,g'$ is a regular sequence on $\widehat{R^+}$ and use the same argument as in the case of $A_\infty/g$ to show that $\widehat{R^+}/g$ is $p$-adically separated.}
\end{proof}

Now we suppose $f/g^m$ is almost integral over $A_\infty$. Then $f/g^m$ is almost integral over $\widehat{A^+}$ and by Theorem~\ref{theorem: completeIntegrallyClosed} and Claim~\ref{claim: p-adic separated}, we have $f\in g^m\widehat{A^+}$. On the other hand, we know that $A_\infty/p^n\to \widehat{A^+}/p^n$ is faithfully flat since $A_\infty$ is the $p$-adic completion of a direct limit of regular local rings and $\widehat{A^+}$ is balanced big Cohen-Macaulay by Theorem~\ref{theorem: BB}. It follows that
$$f\in A_\infty \cap (g^m\widehat{A^+}) \subseteq \cap_n \left((g^m, p^n)\widehat{A^+} \cap A_\infty\right)= \cap_n (g^m, p^n)A_\infty = g^mA_\infty$$
where the last equality follows from Claim~\ref{claim: p-adic separated}. Hence $f/g^m\in A_\infty$, that is, $A_\infty$ is completely integrally closed in its fraction field.
\end{proof}

%\begin{remark}
%Combining Corollary~\ref{corollary: A_infty} and Lemma~\ref{lemma: p adic closure in total integral closure}, we find that principal ideals in $A_\infty$ are $p$-adically closed (or equivalently, $A_\infty/g$ is $p$-adically separated for all $g\in A_\infty$). 
%\end{remark}

Lastly, we construct elements $g\in\widehat{R^+}$ such that $\widehat{R^+}/g$ is not $p$-adically separated whenever $\dim(R)\geq 2$. 

\begin{proposition}
\label{proposition: construction of bad g}
Suppose $\dim(R)\geq 2$ and let $v$ be any $\mathbb{R}$-valuation on $\widehat{R^+}$ whose center on $R^+$ contains $p$ and is not $\m_{R^+}$, normalized so that $v(p)=1$.
Then there exists $0\neq g\in \widehat{R^+}$ such that $v(g)=\infty$ (i.e., $g$ is in the support of $v$).
More precisely, if $x\in\m_{R^+}$ is such that $p,x$ is a regular sequence on $R^+$ and $x$ is not in the center of $v$, then $g$ can be taken as the limit of a sequence $\{g_i\}_{i=1}^\infty$ in $R^+$ so that 
\begin{enumerate}
  %\item $v(g_i)=i+1$
  \item $g_1$ satisfies the equation $T^2+xT+p^2=0$.
  \item $g_i={p^{i+1}}/{s_i}$ for some $s_i\in R^+$ with $v(s_i)=0$ for every $i\geq 1$.
  %SSP
  %
  \item $g_{i+1}=g_i+p^ih_i$ where $h_i$ is a root of the equation $T^2+s_iT+p=0$ for every $i\geq 1$.

\end{enumerate}
In particular, there exists $g\in\widehat{R^+}$ such that $\widehat{R^+}/g$ is not $p$-adically separated.
\end{proposition}

\begin{proof}

For $i=1$, we note that the equation $T^2+xT+p^2=0$ has two roots $u_1,u_2\in R^+$.
Since $v(u_1u_2)=2$ while $v(u_1+u_2)=v(-x)=0$, we may without loss of generality assume $v(u_1)=2$ and $v(u_2)=0$.
%SSP
%
We set $g_1=u_1$ and $s_1=u_2$; so (2) is satisfied for $i=1$.
%As $v(g_1)=2$, (1) is satisfied for $i=1$.
%Further, $g_1=\frac {p^2}{-g_1-x}$, satisfying (2).
%Hence $g_1$ can be chosen to satisfy all conditions.

Assume we have chosen $g_i$ and $s_i$.
We now show we may choose $g_{i+1}$ to satisfy (2) and (3).
By the condition (3), $h_i$ must be one of the two roots of $T^2+s_iT+p=0$, so all we need to show is that one of the two choices will produce a $g_{i+1}$ that satisfies (2).
Now $$g_{i+1}=g_i+p^ih_i=\frac {p^{i+1}}{s_i}+p^{i}h_i=p^i(\frac{p+s_ih_i}{s_i}).$$
Since $h_i$ is one of the two roots of $T^2+s_iT+p=0$, $s_ih_i$ is one of the two roots of $T^2+s_i^2T+ps_i^2=0$ and so $p+s_ih_i$ is one of the two roots of
\begin{align*}
  0 & =  (T-p)^2+s_i^2(T-p)+ps_i^2\\
   & =T^2+(-2p+s_i^2)T+(p^2-ps_i^2+ps_i^2) \\
   & =T^2+(-2p+s_i^2)T+p^2
\end{align*}
We know that one root of the equation above has value $2$ and the other has value $0$ (since $v(s_i)=0$). It follows that we can choose $h_i$ so that $v(p+s_ih_i)=2$
and with this choice, $$p+s_ih_i=\frac{p^2}{-(-2p+s_i^2)-(p+s_ih_i)}$$
%SSP
%
and the denominator above (which is the other root) has value $0$. Setting $s_{i+1}$ to be $s_i$ times the denominator above, we have $v(s_{i+1})=0$ and $g_{i+1}={p^{i+2}}/{s_{i+1}}$ as wanted.

We have constructed the desired sequence $\{g_i\}_{i=1}^\infty$ in $R^+$. By (3), the limit $g=\lim_ig_i$ is well-defined as an element in $\widehat{R^+}$. By (2), we know that 
$v(g_i)=i+1$ and so $v(g)=\infty$. By (1), we have $g_1\notin pR^+$: If $g_1=py$, then we have $p^2y^2+pxy+p^2=0$ and thus $xy\in pR^+$, so $y\in pR^+$ as $p, x$ is a regular sequence in $R^+$. But then $g_1=p^2z$ and thus $p^4z^2+p^2xz+p^2=0$, so $p^2z^2+xz+1=0$ a contradiction. Therefore, $g\notin p\widehat{R^+}$ and in particular, $g\neq 0$.
 The last conclusion follows by taking $v=v_Q$ and applying Theorem~\ref{theorem: completeIntegrallyClosed}. 
\end{proof}

%Specifically, one can take a height one prime $Q$ in $R^+$ and consider the valuation ring $V^+:=R^+_Q$, and let $v_Q$ be the valuation on $\widehat{R^+}$ induced by the canonical map $\widehat{R^+}\to \widehat{V^+}$. The conclusion of the proposition then says that $g$ is in the kernel of the canonical map $\widehat{R^+}\to \widehat{V^+}$. By Theorem~\ref{theorem: completeIntegrallyClosed}, $\widehat{R^+}/g$ is not $p$-adically separated.

\begin{remark}
In fact, when $\dim(R)\geq 2$, we can show that the zero ideal in $\widehat{R^+}$ in not the support of any rank one valuation $v$ whose center contains $p$. By Proposition~\ref{proposition: construction of bad g}, it is enough to handle the case that $v$ is centered on $\m_{R^+}$. We may normalize $v$ so that $v(p)=1$ and choose $x\in R^+$ so that $p,x$ is a regular sequence on $R^+$, without loss of generality, we may also assume that $v(x)>1$. Then by applying Lemma~\ref{lemma: constructing h} below repeatedly and taking limit, we obtain an element $g\in \widehat{R^+}$ so that $g\equiv x$ mod $p$ (in particular, $g\neq 0$) and $v(g)=\infty$.  This fact shows, in the language of \cite[Remark 2.14, Definitions 2.16 and 2.18]{DineTopologicalSpectrum}, that the zero ideal is a spectrally reduced prime ideal of the Tate algebra $S:=\widehat{R^+}[1/p]$ (i.e., $(0)\in \Spec_{\text{Top}}(S)$) that is not the support of any valuation in the Berkovich space $\mathcal{M}(S)$. 
\end{remark}

\begin{lemma}
\label{lemma: constructing h}
Suppose $v$ is any $\mathbb{R}$-valuation on $\widehat{R^+}$ centered on the maximal ideal, normalized so that $v(p)=1$. If $x\in \m_{R^+}$ is such that $v(x)=\lambda>i\geq 1$. Then there exists $h\in R^+$ so that $v(x+p^ih)>i+1$.  
\end{lemma}
\begin{proof}
Consider the monic polynomial $T^N + p^iT^2 +xT +x^N.$
If $h\in R^+$ is a root of this polynomial, then $x+p^ih$ is a root of the polynomial 
$$f(T):=(T-x)^N + p^{(N-1)i}(T-x)^2 + p^{(N-1)i}x(T-x) +p^{Ni}x^N.$$
Writing $f(T)=T^N + b_{N-1}T^{N-1}+ \cdots + b_1T + b_0$, it is straightforward to check that 
$$b_1 = N(-x)^{N-1} -p^{(N-1)i}x \,\ \text{ and } \,\  b_0= (-x)^N+ p^{Ni}x^N.$$
It follows that $v(b_0)= N\lambda$ and $v(b_1)=(N-1)i+\lambda$ for $N\gg0$. Thus for $N\gg0$,
$$v(b_1/b_0)=(N-1)(i-\lambda)< -i-1.$$
Thus there exists an $h$ so that $v(1/(x+p^ih))< -i-1$, that is, $v(x+p^ih)> i+1$. 
\end{proof}

Finally, we can prove Theorem A from the introduction.

\begin{theorem}
\label{theorem: completeIntegrallyClosedShortVersion}
Let $(R,\m)$ be a Noetherian complete local domain of mixed characteristic $(0,p)$ with fraction field $K$ and its algebraic closure $\overline{K}$. Then $\widehat{R^+}$ is completely integrally closed in $\widehat{R^+}\otimes_{R^+}\overline{K}$, while $\widehat{R^+}$ is completely integrally closed in its own fraction field if and only if $\dim(R)=1$. 
\end{theorem}
\begin{proof}
If $\dim(R)=1$, then $\widehat{R^+}$ is a valuation ring of rank one. In particular $\widehat{R^+}$ is completely integrally closed in its fraction field. Suppose $\dim(R)\geq 2$. If $g\in R^+$, then condition $(2)$ in Theorem~\ref{theorem: completeIntegrallyClosed} is obvious and the theorem follows from Theorem~\ref{theorem: completeIntegrallyClosed} and Proposition~\ref{proposition: construction of bad g}.
\end{proof}

In connection with Theorem~\ref{theorem: completeIntegrallyClosedShortVersion}, the following question is very natural to ask, and we do not know the answer.

\begin{question}
Let $(R,\m)$ be a Noetherian complete local domain of mixed characteristic $(0,p)$ with $\dim(R)\geq 2$. Is $\widehat{R^+}$ integrally closed in its fraction field?
\end{question}

\bibliographystyle{skalpha}
\bibliography{CommonBib}
\end{document}